\documentclass[11pt]{amsart}
\usepackage{graphicx}
\usepackage{epstopdf}
\usepackage{url}
\usepackage{latexsym}
\usepackage{amsfonts,amsmath,amssymb}
\newtheorem{theorem}{Theorem}[section]

\newtheorem{lemma}[theorem]{Lemma}

\newtheorem{corollary}[theorem]{Corollary}

\theoremstyle{remark}
\newtheorem{remark}{Remark}

\date{\today}

\begin{document}

\title[Products of Commutators of Finite Order Elements]{Expressing Finite-Infinite Matrices Into Products of Commutators of Finite Order Elements}

\author{Ivan Gargate}
\address{UTFPR, Campus Pato Branco, Rua Via do Conhecimento km 01, 85503-390 Pato Branco, PR, Brazil}
\email{ivangargate@utfpr.edu.br}

\author{Michael Gargate}
\address{UTFPR, Campus Pato Branco, Rua Via do Conhecimento km 01, 85503-390 Pato Branco, PR, Brazil}
\email{michaelgargate@utfpr.edu.br}

\begin{abstract}

Let $R$ be an associative ring with unity $1$ and consider $k\in \mathbb{N}$ such that $1+1+..+1=k$ is invertible.  Denote by $\omega$ an arbitrary kth root of unity in $R$ and let $UT^{(k)}_{\infty}(R)$ be the group of upper triangular infinite matrices whose diagonal entries are $k$th roots of $1$. We show that every element of the group $UT_{\infty}(R)$ can be expressed as a product of $4k-6$ commutators all depending of powers of elements in $UT^{(k)}_{\infty}(R)$ of order $k$. If $R$ is the complex field or the real number field we prove that, in $SL_n(R)$ and in the subgroup $SL_{VK}(\infty,R)$ of
the Vershik-Kerov group over $R$, each element in these groups can be decomposed into a product of at most $4k-6$ commutators of elements of order $k$.
\end{abstract}



\maketitle

\section{Introduction}\label{intro}
Expressing matrices as a product of involutions was studied by several authors a few years ago. For example Halmos in \cite{Halmos} proved that every square matrix over a field, with determinant $\pm 1$, is the product of not more than four involutions and Solwik in \cite{Slowik-1} proved that for any field, every element of the group  of upper triangular infinite matrices whose entries lying on the main diagonal are
equal to either $1$ or $-1$ can be expressed as a product of at most five involutions.

Following the same direction, there are works to express matrices as the product of commutators of matrices. For example Zheng  in \cite{Zheng} proved that every matrix $A$ in $SL_n(F)$ is a product of at most two commutators of involutions, where $F$ is the complex number field or the real number field and  Hou in \cite{Hou} proved that the group of upper triangular infinite matrices whose entries lying on the main diagonal are equal to 1 can be expressed as a product of at most two commutators of involutions.
 
 Recently Slowik in \cite{Slowik2} and Grunenfelder in \cite{Grunenfelder} study when a matrix can be expressed as a product of fixed order matrices. 
 
 In this paper, the authors generalize the work done by Hou in \cite{Hou} about the group $ UT^{(k)}_n(R)$ and $UT^{(k)}_{\infty}(R)$ of matrices whose elements in $T_n(R)$ and $T_{\infty}(R)$ respectively have entries in the diagonal of order $k$,  obtaining results of when a matrix is the product of commutators of matrices  of a fixed order.

 The main result of this paper is stated as follows:
 
 \begin{theorem}\label{th1}
Assume that $R$ is an associative ring with unity $1$ and that $1+1+\cdots+1=k$ is an invertible element of $R$. Then every element of the group $UT_{\infty}(R)$ and $UT_n(R)$ $(n\in \mathbb{N})$ can be expressed as a product of at most $4k-6$ commutators of elements of order $k$ in $UT^{(k)}_{\infty}(R)$ and $UT^{(k)}_n(R)$, respectively.
\end{theorem}
 
 In the section 3, consider the case $R=\mathbb{K}$ a complex field or the real field and the group $SL_n(\mathbb{K})$ we have the following result:
 
 \begin{theorem}\label{th2}
All element in $SL_n(\mathbb{K})$ can be written as a product of at most $4k-6$ commutators of elements of order $k$ in $GL_n(\mathbb{K}).$
\end{theorem}
 
And, if we consider $GL_{VK}(\infty,\mathbb{K})$ the Vershik-Kerov group, we have:

\begin{theorem}\label{th3} Assume that $\mathbb{K}$ is a complex field or the real number field. Then every element of the
group $SL_{VK}(\infty, \mathbb{K})$ can be expressed as a product of at most $4k-6$ commutators of elements of order $k$ in
$GL_{VK}(\infty, \mathbb{K})$.
\end{theorem}

\section{Preliminaries}

Fix $k\in \mathbb{N}$, $k\geq 2$ and let $R$ be an associative ring with unity and denote by $\omega$ an arbitrary kth root of unity in $R$. Denote by $T_n(R)$ and $T_{\infty}(R)$ the groups of $n\times n$ and infinite upper triangular matrices over a ring $R$, respectively.   Anagolously, denote by $UT_n(R)$ and $UT_{\infty}(R)$ the groups of upper triangular matrices whose entries on the main diagonal are equal to unity $1$. We also put

$$UT^{(k)}_n(R)=\{g\in T_n(R); \ g_{ii}^k=1\}.$$
$$UT^{(k)}_{\infty}(R)=\{g\in T_{\infty}(R); \ g_{ii}^k=1\}.$$
$$D_n^{(k)}(R)=\{g\in UT^{(k)}_n(R); \ g_{ij}=0, \ if \ i\neq j\}.$$
$$D^{(k)}_{\infty}(R)=\{g\in UT^{(k)}_{\infty}(R); \ g_{ij}=0, \ if \ i\neq j\}.$$

Denote by $E_{ij}$ the finite or infinite matrix with a unique nonzero entry equal to $1$ in the position $(i,j)$, so  $A=\sum_{1\leq i\leq j \leq n}a_{i,j}E_{ij}$ ($A=\sum_{i,j\in \mathbb{N}}a_{ij}E_{i,j}$) is the $n\times n$ (infinite $\mathbb{N}\times \mathbb{N}$) matrix with $a_{ij}$ in the position $(i,j)$. Denote by $I_n$ and $I_{\infty}$ the identity matrices in $T_n(R)$ and $T_{\infty}(R)$, respectively.
The following remark is immediate:

\begin{remark} Let $G$ a group
\begin{itemize}
    \item [1,] If $g\in G$ is an element of order $k$, i.e. $g^k=1$, then for every $h\in G$ the conjugation $g^h=hgh^{-1}$ is an element of order $k$.
    \item[2.] If $g\in G$ is a product of $r$ element of order $k$, then for every $h\in G$ the conjugation $g^h$ is a product of $r $ elements of order $k$.
\end{itemize}
\end{remark}

Define the commutator $[\alpha,\beta]=\alpha\beta\alpha^{-1}\beta^{-1}.$ Denote by $E_{ij}$ the finite or infinite matrix  with an unique nonzero entry equal to 1 in the position $(i,j)$. Then, if $A\in UT^{(k)}_n(R)$ or $UT^{(k)}_{\infty}(R)$ we can writte
$$A=\sum_{i,j}a_{i,j}E_{i,j}.$$

Denote by $J_{\infty}(R)$ the set of all matrices in $T_{\infty}(R)$ in which all 
entries outside of the first superdiagonal equal $0$. Let $A\in UT^{(k)}_{\infty}(R)$ and denote by $J(A)$ the matrix of $J_{\infty}(R)$ that has the same entries on the first super diagonal as $A$. Denote by $Z$ the center of the ring $R$, and by $D_{\infty}(Z)$ the subring of all diagonal infinite matrices with entries in $Z$.
\newline
We say that $A$ is $coherent$ when there is a sequence $(D_k)_{k\geq 1}$ of elements of $D_{\infty}(Z)$ such that
$$A=\displaystyle \sum^{\infty}_{k=0}D_kJ(A)^k.$$

In such a sequence we can always require $D_0=D_1=I_{\infty}$, in which case we say that the sequence $(D_k)_{k\geq 1}$ is $normalized$.
\\
\section{Expressing Matrices into Products of Commutators}
Assume that $2$ is  an invertible element of $R$. The following results are of Hou in \cite{Hou}:
\begin{lemma}\label{lemma1} If $A\in UT_{\infty}(R)$ is coherent, then $A^2$ is coherent.
\end{lemma}

\begin{lemma}\label{lemma2} Assume that $R$ is an associative ring with unity $1$ and that $2$ is an invertible
element of $R$. Let $J \in J_{\infty}(R)$. Then, there exists a coherent matrix $A\in UT_{\infty}(R)$ such that
$J(A) = J$ and $A$ is the commutator of two involutions in $T_{\infty}(R)$.
\end{lemma}

and 

\begin{lemma}\label{lemma3}
Assume that $R$ is an associative ring with unity $1$ and that $2$ is an invertible
element of $R$. Let $A, B$ be coherent matrices of $UT_{\infty}(R)$ such that $J(A) = J(B)$. Then, $A$
and $B$ are conjugated in the group $UT_{\infty}(R)$.
\end{lemma}

The following result generalized the Lemma \ref{lemma1}

\begin{lemma}\label{lema4}
If $A\in UT_{\infty}(R)$ is coherent, then $A^k$ is coherent, for all $k\in \mathbb{N}$ such that $k$ is invertible in $R$.
\end{lemma}

\begin{proof} The case for $k=2$ has been demonstrated by Hou in \cite{Hou}, here show that if $A=\sum^{\infty}_{i=0}D_iJ(A)^i$ with $D_0=D_1=I_{\infty}$ and $D_2,D_3,\cdots \in D_{\infty}(Z)$, then
$$A^2= \sum^{\infty}_{i=0}\left(\sum^{i}_{j=1}D_jS^j(D_{i-j})\right)J(A)^i,$$
where, if $D\in D_{\infty}(Z)$ has diagonal entries $d_1,d_2,\cdots, d_k, \cdots$, then define $S(D)\in D_{\infty}(Z)$ as the matrix with diagonal entries $d_2,d_3,\cdots,d_{k+1},\cdots.$  \newline
Denote by $D^{(2)}_i= \sum^{i}_{j=1}D_jS^j(D_{i-j})$. We can write 
$$A^2= \sum^{\infty}_{i=0}D^{(2)}_{i}J(A)^i,$$
this follows that
$$A^3=\sum^{\infty}_{i=0}\left(\sum^{i}_{j=0}D_jS^j(D^{(2)}_{i-j})\right)J(A)^i,$$
and, in geral if 
$$A^l=\sum^{\infty}_{i=0}D^{(l)}_iJ(A)^i,$$
where $D^{(l)}_0=D^{(l)}_1=I_{\infty}$ and $D^{(l)}_2,D^{(l)}_3,\cdots \in D_{\infty}(R)$, then
$$A^{l+1}=\sum^{\infty}_{i=0}\left(\sum^{i}_{j=0}D_jS^j(D^{(l)}_{i-j})\right)J(A)^i.$$
Note that $J(A)^3=(D_0 D^{(2)}_1+D_1S(D^{(2)}_0))J(A)=(2+1)J(A)=3J(A)$, in geral $J(A)^l=lJ(A)$ for $l\geq 2$ and using the invertibility of $k$ we conclude that $A^k$ is coherent.
\end{proof}

The following is an adaption of the Lemma \ref{lemma2} to our case.

\begin{lemma}\label{lema5} Assume that $R$ is an associative ring with unity $1$ and that $1+1+\cdots+1=k$ is an invertible element of $R$. Let $J\in J_{\infty}(R)$. Then, there exists a coherent matrix $A\in UT_{\infty}(R)$ such that $J(A)=J$ and $A$ is the product of $2k-3$ commutators all depending on  two matrices $B,C$ in $UT^{(k)}_{\infty}(R)$ of order $k$.
\end{lemma}

First we proof the following:

\begin{lemma}\label{lema6} Suppose that $k>1$ and that $B,C$ are matrices of order $k$. Then $(BC)^k$ can be represented as products of $2k-3$ commutators all depending on the powers of $B$ and $C$.
\end{lemma}

\begin{proof} Define by recurrence:
$$F_2(B,C)=[B,C],$$ $$F_3(B,C)=F_2(B,C)[C,B^2][B^2,C^2]$$ and $$F_{i+1}(B,C)=F_{i}(B,C)[C^{(i-1)},B^i][B^i,C^i],$$ the we have  for example:

$$\begin{array}{rl} F_2(B,C)=& [B,C] \\ F_3(B,C)=&[B,C][C,B^2][B^2,C^2] \\ F_4(B,C)=&F_3(B,C)[C^2,B^3][B^3,C^3] \\
=&[B,C][C,B^2][B^2,C^2] [C^2,B^3][B^3,C^3]\\
F_5(B,C)=& F_4(B,C)[C^3,B^4][B^4,C^4] \\ 
 =&[B,C][C,B^2][B^2,C^2] [C^2,B^3][B^3,C^3][C^3,B^4][B^4,C^4], \end{array}$$
then we affirmate that 
\begin{equation}\label{eqcom}
(BC)^i=F_i(B,C)C^{i-1}B^iC.
\end{equation}
Its clear that this is true for $i=2$. Suppose that this is true for $i=h$, then, multiplicating by $(BC)$ in both sides of (\ref{eqcom}) we have:
$$\begin{array}{rl}
    (BC)^{h+1} &=(BC)^h(BC)  \\
     & =(F_h(B,C)C^{h-1}B^hC)(BC)\\
     &= F_h(B,C))C^{h-1}B^h(C^{-(h-1)}B^{-h}B^{h}C^{(h-1)})CBC \\
     &= F_h(B,C)[C^{h-1},B^h]B^{h}C^hBC \\
     &= F_h(B,C)[C^{h-1},B^{h}]B^{h}C^hB^{-h}C^{-h}C^hB^{h}BC \\
     &=F_h(B,C)[C^{h-1},B^{h}][B^{h},C^h]C^hB^{h}BC  \\
     &=F_{h+1}(B,C)C^hB^{(h+1)}C.
\end{array}$$
This proof the our affirmation (\ref{eqcom}). In the case that $i=k$ then we have that
$$(BC)^k=F_k(B,C)C^{k-1}B^kC=F_k(B,C),$$
this finalizing the proof.

\end{proof}
\begin{proof}[Proof of Lemma \ref{lema5}]
 Let $J=\displaystyle\sum^{\infty}_{i=1}a_{i,i+1}E_{i,i+1}\in J_{\infty}(R)$ and define
$$B=\left(\begin{array}{cccccc} 1 & 0 &  &  &  & \cdots \\ 0 & w & \frac{1}{k}a_{23} & &  &  \\  &  & 1& 0 &  &  \\  & & & w  & \frac{1}{k}a_{45} &  \\  &  &  &  &  & \ddots \end{array}\right)=diag \left(1,\left(\begin{array}{cc}w & \frac{1}{k}a_{23} \\ 0 & 1 \end{array}\right),\left(\begin{array}{cc}w & \frac{1}{k}a_{45} \\ 0 & 1 \end{array}\right),\cdots\right),$$
$$ \ \ = \sum^{\infty}_{i=1}\left(E_{2i-1,2i-1}+ wE_{2i,2i}+\frac{1}{k}a_{2i.2i+1}E_{2i,2i+1} \right) $$

$$C=\left(\begin{array}{cccccc} 1 &\frac{1}{k}a_{12} &  &  &  & \cdots \\ & w^{-1} & 0 & &  &  \\  &  & 1 & \frac{1}{k}a_{34} &  &  \\  & & & w^{-1}  &  &  \\  &  &  &  &  & \ddots \end{array}\right)=diag \left(\left(\begin{array}{cc}1 & \frac{1}{k}a_{12} \\ 0 & w^{-1} \end{array}\right),\left(\begin{array}{cc}1& \frac{1}{k}a_{34} \\ 0 & w^{-1} \end{array}\right),\cdots\right),$$

$$ \ \ = \sum^{\infty}_{i=1}\left(E_{2i-1,2i-1}+ w^{-1}E_{2i,2i}+\frac{1}{k}a_{2i-1,2i}E_{2i-1,2i} \right). $$
Observe that $B$ and $C$ are two matrices of order $k$. Using the Lemma \ref{lema6}, denote by $A=(BC)^k=F_k(B,C)$
then by lemma \ref{lema4} $J(A)=J((BC)^k)=kJ(BC)$. Observe that

$$BC=I_{\infty}+\frac{1}{k}\sum^{\infty}_{i=1}a_{i,i+1}E_{i,i+1}+\frac{1}{k^2}\sum^{\infty}_{i=1}a_{2i,2i+1}a_{2i+1,2i+2}E_{2i,2i+2}$$
And $BC$ is coherent with $D_0=D_1=I_{\infty}$, and $D_2=\sum^{\infty}_{i=1}E_{2i,2i}.$ This conclude the lemma.
\end{proof}


From the lemmas \ref{lemma3}, \ref{lema4} and \ref{lema5}, we obtain the following corollary


\begin{corollary}\label{corol77}
Assume that $R$ is an associative ring with unity $1$ and that $1+1+\cdots+1=k$ is an invertible element of $R$. Every matrix in $UT_{\infty}(R)$ and $UT_n(R)$ $(n\in \mathbb{N})$, whose  entries except in the main diagonal and the first super diagonal are all equal to zero, is a product of $2k-3$ commutators of power of two elements of order $k$ in  $UT^{(k)}_{\infty}(R)$.
\end{corollary}

\begin{proof}
Let $A\in UT_{\infty}(R)$, then by the Lemma \ref{lema5}, for $J(A)$ there is an element $B\in UT_{\infty}(R)$ such that $J(B)=J(A)$ and $B$ is the product of $2k-3$ commutators of powers of two elements of order $k$ in $UT^{(k)}_{\infty}(R)$ and by Lemma \ref{lemma3} we conclude that $A$ and $B$ are conjugated.  
\end{proof}

We enunciated the followings results adapted from \cite{Hou}:

\begin{lemma}\label{matrix}
Assume that $R$ is an associative ring with unity $1$ and let $n\in \mathbb{N}$:
\begin{itemize}
\item[1.] If $A,B$ are elements of $UT^{(k)}_{n}(R)$ such that $a_{i,i+1}=b_{i,i+1}=1$ for all $1\leq i\leq n-1,$ then $A$ and $B$ are conjugated in $UT^{(k)}_{n}(R)$.
\item[2.] If $A,B$ are elements of $UT^{(k)}_{\infty}(R)$ such that $a_{i,i+1}=b_{i,i+1}=1$ for all $1\leq i$ then $A$ and $B$ are conjugated in $UT^{(k)}_{\infty}(R)$.
\end{itemize}
\end{lemma}
\begin{proof}
Consider $A=(a_{ij})$ and $$J=\left(\begin{array}{ccccc} a_{11} &1 &  &   & \\ & a_{22} & 1 & & \\  &  & a_{33} & 1  &    \\ & &  &\ddots & \ddots    \end{array}\right).$$
All blank entries are equal to $0$. We need only to prove that $A$ is conjugated to the matrix $J$, for this we can constructed a matrix $X=(x_{ij})\in UT^{(k)}_{\infty}(R)$ such that $X^{-1}AX=J$ or $AX=XJ$. For the first super diagonal entries of $X$ we choose 
\begin{itemize}
    \item [a.)] In the case that $a_{ii}=a_{i+1,i+1}$ then $x_{i,i+1}$ can be choose any element in $R$.
    \item[b.)] In the case that $a_{ii}\neq a_{i+1,i+1}$ then $x_{i,i+1}= 1$.
\end{itemize}
When we have first for the second super diagonal entrieswe can choose
\begin{itemize}
    \item [a.)] If $a_{ii}=a_{i+2,i+2}$ the $x_{i,i+2}$ can be choose any element in $R$.
    \item[b.)] In the case that $a_{ii}\neq a_{i+2,i+2}$ then $$x_{i,i+2}=[a_{ii}-a_{i,i+2}]^{-1}\{x_{i,i+1}-x_{i+1,i+2}-a_{i,i+2}x_{i+2,i+2}\}.$$
\end{itemize}
When we have first, second, $\cdots$, $(j-1)-th$ super diagonal entries, to obtain the element $x_{ij}$ he can choose

\begin{itemize}
    \item [a.)] If $a_{ii}=a_{jj}$ then choose $x_{ii}=x_{jj}$ and $x_{ij}$ can be choose any element in $R$.
    \item[b.)] If $a_{ii}\neq a_{jj}$ then choose $x_{ii}\neq x_{jj}$ and $$x_{ij}=[a_{ii}-a_{jj}]^{-1}\{ x_{i,j-1}-\displaystyle\sum^{k=j-i}_{k=2} a_{i,i+k}x_{i+k,j}\}$$
\end{itemize}
Thus the Lemma \ref{matrix} is proved.
\end{proof}

Also:


\begin{corollary}\label{corol79}
Assume that $R$ is an associative ring with unity $1$ and that $1+1+\cdots+1=k$ is an invertible element of $R$. Every matrix in $UT_{\infty}(R)$ and $UT_n(R)$ $(n\in\mathbb{N})$, whose entries in the first super diagonal are all equal to the unity $1$, is a product of $2k-3$ commutators of powers of two elements of order $k$ in $UT^{(k)}_n(R)$ or $UT^{(k)}_{\infty}(R)$ .
\end{corollary}
\begin{proof}Consider $A\in UT_{\infty}(R)$, then by the Lemma \ref{matrix} $A$ is conjugated to matrix $J=J(A)$, and by the Corollary \ref{corol77}, $J$ is product of $2k-3$ commutators of matrices of order $k$. 
\end{proof}

Then, with this results, we show the Theorem \ref{th1}:

\begin{proof}[Proof of Theorem \ref{th1}]
Let $A$ be an arbitrary matrix of $UT_{\infty}(R)$. We can write
$$A=I_{\infty}+ \displaystyle \sum^{\infty}_{i=1} \sum^{\infty}_{j=i+1}a_{i,j}E_{i,j} \ \in UT^{(k)}_{\infty}(R),$$
and consider the following matrix

$$B=I_{\infty}+\sum^{\infty}_{i=1}(a_{i,i+1}-1)E_{i,i+1} \in UT_{\infty}(R).$$
By Corollary \ref{corol77}, $B$ is a product of $2k-3$ commutators of two matrices of order $k$ in $UT^{(k)}_{\infty}(R)$. Note that 
$$C=B^{-1}A\in UT_{\infty}(R),$$ is a matrix whose entries in the main diagonal and the first super diagonal are all equal to $1$. By the Corollary  \ref{corol79}, $C$ is also a product of $2k-3$ commutators of powers of elements of order $k$ in $UT^{(k)}_{\infty}(R)$. Finally, $A=BC$ is a product of $4k-6$ commutators of powers of four elements of order $k$ in $UT^{(k)}_{\infty}(R)$.

\end{proof}

\section{Case $R=\mathbb{K}$ a complex field or the real field}

Consider $R=\mathbb{K}$ be a complex field or the real number field, then we show the Theorem \ref{th2}: 

\begin{proof}[Proof of Theorem \ref{th2}]
The case $k=2$ is proved in \cite{Hou}. Suppose that $k\geq 3$. Consider $A\in SL_n(\mathbb{K}) $ not a scalar matrix, then by \cite{Sourour}, Theorem 1, we can find a lower-triangular matrix $L$ and a upper-triangular matrix $U$ such that $A$ is similar to $LU$, and both $L$ and $U$ with all entries on the main diagonal equal to 1. By the corollary \ref{corol79} it follows that each of the matrices $L$ and $U$ is a product of $2k-3$ commutators of powers of two elements of order $k$ from $SL_n(\mathbb{K})$. The scalar case $A=\alpha I$ with $det(A)=1$ it suffices to consider the case when $n$ is exactly the order of $\alpha$ (see \cite{Hou}). Next we use the ideas from Grunenfelder \cite{Grunenfelder}. Observe that for $a\in \mathbb{K}$, $a\notin \{0,1,-1\}$ then
 $$\left[\begin{array}{cc}a & 0 \\ 0 & a^{-1}\end{array}\right]= J_1(a)\cdot J_2(a)$$
 with
 $$J_1(a)=\left[\begin{array}{cc}\displaystyle\frac{at}{a+1} & a-\left(\displaystyle\frac{at}{a+1}\right)^2 \\
 -\displaystyle\frac{1}{a} & \displaystyle\frac{t}{a+1}\end{array}\right] \ and \  J_2(a)=\left[\begin{array}{cc}\displaystyle\frac{at}{a+1} & a\left(\displaystyle\frac{at}{a+1}\right)^2-1 \\ 
		1 & \displaystyle\frac{t}{a+1}\end{array}\right],$$
 where $t=\theta+\theta^{-1}$, $\theta^k=1$, $\theta\neq 1$. We have that $J_1(a)^k=J_2(a)^k=I$. Then, if $n$ is even we can express $\alpha I$ as a product of two matrices:

$$\alpha I=\left[\begin{array}{cccccc}
\alpha &             & & & & \\
       & \alpha^{-1} & & & & \\
  & & \alpha^3 & &  & \\
  & & &  \ddots & & \\
   & & &         &  \alpha^{n-1} & \\
    & & &        &  & \alpha^{-n+1}\end{array}\right] \left[\begin{array}{cccccc}
1 &             & & & & \\
       & \alpha^{2} & & & & \\
  & & \alpha^{-2} & &  & \\
  & & &  \ddots & & \\
   & & &         &  \alpha^{-n+2} & \\
    & & &        &  & 1\end{array}\right].$$
 
 And in the case that $n$ is odd
 
 $$\alpha I=\left[\begin{array}{cccccc}
\alpha &             & & & & \\
       & \alpha^{-1} & & & & \\
  & & \alpha^3 & &  & \\
  & & &  \ddots & & \\
   & & &         &  \alpha^{-n+2} & \\
    & & &        &  & 1\end{array}\right] \left[\begin{array}{cccccc}
1 &             & & & & \\
       & \alpha^{2} & & & & \\
  & & \alpha^{-2} & &  & \\
  & & &  \ddots & & \\
   & & &         &  \alpha^{n-1} & \\
    & & &        &  & \alpha^{-n+1}\end{array}\right].$$
 
 Denote by $F$ and $G$ the first and second matrix that appears in these decompositions. If $n$ is even, for example, $F=diag(\alpha,\alpha^{-1},\alpha^3,\cdots, \alpha^{n-1},\alpha^{-n+1})$ and define the matrix $S=diag(\alpha^{1/k},\alpha^{-1/k},\alpha^{3/k},\cdots, \alpha^{(n-1)/k},\alpha^{(-n+1)/k}).$ Then, each block $2\times 2$ of the form $diag(\alpha^{j/k},\alpha^{-j/k})$ is product of $J_1(\alpha^{j/k})$ and $J_2(\alpha^{j/k})$, both of order $k$, so define $J_1$ and $J_2$ as the block diagonal matrices with entries $J_1(\alpha^{j/k})$ and $J_2(\alpha^{j/k})$ respectively, then $F=S^k=(J_1\cdot J_2)^k$. By the Lemmas \ref{lema6} and \ref{lema5} we have that $F$ is product of $2k-3$ commutators of powers of $J_1$ and $J_2$. Similarly we can obtain the same results for the matrix $G$. Therefore, by the decomposition, we conclude that $\alpha I_n$ is a product of $4k-6$ commutators. The case that $n$ is odd is analogue.
\end{proof}

Consider the $GL_{VK}(\infty,\mathbb{K})$ the Vershik-Kerov group consisting of all infinite matrices of the form
\begin{equation}\label{eq111}
\left(\begin{array}{c|c}M_1 & M_2 \\ \hline 0 & M_3 \end{array}\right)
\end{equation}
where $M_1\in GL(n,\mathbb{K})$ for some $n\in \mathbb{N}$ and $M_3\in T_{\infty}(\mathbb{K}).$ 

And we can use the following lemma (see \cite{Hou}):

\begin{lemma}
Assume that $\mathbb{K}$ is a complex field or the real number field. Let $A\in GL_n(\mathbb{K})$ of which $1$ is no eigenvalue, and let $T$ be an infinite unitriangular matrix. In the Vershik-Kerov group, any matrix of the form 
\begin{equation}\label{eq112}
\left(\begin{array}{c|c}A & B \\ \hline 0 & T \end{array}\right)
\end{equation} 
is conjugated to 
\begin{equation}\label{eq113}
\left(\begin{array}{c|c}A & 0 \\ \hline 0 & T \end{array}\right)
\end{equation}
\end{lemma}

Then, we shown the Theorem \ref{th3}

\begin{proof}[Proof of Theorem \ref{th3}]  Consider $M\in SL_{VK}(\mathbb{K})$ in the form $M=\left(\begin{array}{c|c} M_1 & M_2  \\ \hline 0 & M_3\end{array}\right)$, with $M_1\in SL_n(\mathbb{K})$ and $M_3\in UT_{\infty}(\mathbb{K})$. From the proof of Theorem 1.3 in \cite{Hou}, $M$ is conjugated to an infinite matrix of the form $\left(\begin{array}{c|c} A & 0  \\ \hline 0 & T\end{array}\right)$, with $A\in SL_n(\mathbb{K})$ for which $1$ is no eigenvalue and $T \in UT_{\infty}(\mathbb{K})$. By the Theorem \ref{th1} and the Theorem \ref{th2}, both are products of $4k-6$ commutators of elements of order $k$ and we know that the direct sum of $A$ and $T$ is also a product of $4k-6$ commutators of elements of order $k$ then this shows that $M$ is also a product of $4k-6$ commutators of elements of order $k$.
\end{proof}

\end{document}